\documentclass[a4paper,12pt]{article}
\usepackage[usenames,dvipsnames]{pstricks}
 \usepackage{epsfig}
 \usepackage{pst-grad} 
 \usepackage{pst-plot} 

\RequirePackage[OT1]{fontenc}
\RequirePackage{amsthm,amsmath}
\RequirePackage[numbers]{natbib}
\RequirePackage[colorlinks,citecolor=blue,urlcolor=blue]{hyperref}
\usepackage{amsmath,amsthm,amsfonts,amssymb}

\numberwithin{equation}{section}
 \theoremstyle{plain} 
\newtheorem{thm}{Theorem}
\newtheorem{conj}{Conjecture}
\newtheorem{cor}[thm]{Corollary}

\newtheorem{prop}{Proposition}

\newtheorem{ass}{Assumption}
\theoremstyle{definition}

\theoremstyle{remark}
\newtheorem{rem}{Remark}[section]

\def\orig{o}

\newcommand{\G}{\mathcal{G}}

\newcommand{\Tr}{\mathtt{Tr}}

\newcommand{\N}{\mathbb{N}}
\newcommand{\T}{\mathbb{T}}
\newcommand{\1}{\textbf{1}}
\newcommand{\ro}{\textbf{r}}
\renewcommand{\P}{\mathbb{P}}

\newcommand{\E}{\mathbb{E}}

\newcommand{\GW}{\textbf{GW}}

\newcommand{\UGW}{\textbf{UGW}}

\newcommand{\eval}[2][\right]{\relax
  \ifx#1\right\relax \left.\fi#2#1\rvert}

\begin{document}

\title{Ends of branching random walks on planar hyperbolic Cayley graphs}
\author{Lorenz A.~Gilch and {Sebastian}  {M\"uller}}
\maketitle



 

 {\abstract We prove that the trace of a transient branching random walk on a planar hyperbolic Cayley graph has a.s.~continuum many ends and no isolated end.}
 ~

 

\section{Introduction}\label{sec:intro}

A \emph{branching random walk} (BRW) is a growing cloud of particles on some graph $G$ in discrete time. The process starts with one particle in the root $\orig$ of the graph. At each time step each particle splits into offspring particles which then move one step according to a random walk on $G$. Particles branch and move independently of the other particles and the history of the process. 
A BRW is therefore driven by two classical stochastic processes: Galton-Watson processes and random walks. Under the assumption that the underlying Galton-Watson process survives the number of particles grows exponentially. If the return probabilities of the underlying random walk decay subexponentially the effect of the growing particles overshadows the transience of the spatial dynamic and the underlying graph will eventually be full of particles. However, if  the return probabilities of  the random walk decay exponentially as well, then there is a critical  growth rate of the Galton-Watson process  where the two exponential effects cancel out.
While above the critical value the BRW is again \emph{recurrent}, i.e.~every finite set is visited infinitely many times, below and at the critical value every finite set is eventually free of particles and the BRW is called \emph{transient}. In the transient case, the set of visited vertices and traversed edges defines a proper random subgraph of $G$ and its properties become of interest. This subgraph is called the \emph{trace} $\Tr$ of the BRW.

Let  $\mu=(\mu_k)_{k\geq 0}$ be the probability distribution that describes  the branching, i.e.~each particle produces $k$ offspring with probability $\mu_{k}$. 
The expected number of offspring is denoted by $m=\sum_{k\geq 0} k \mu_{k}$. 
In this note we assume the underlying graph $G=G(\Gamma,S)$ to be the Cayley graph of a finitely generated group $\Gamma$ with respect to a finite symmetric generating set $S$. The movement of the particles is driven by a driving measure $q$ on $S\cup \{e\}$, where $e$ is the group identity. The driving measure $q$ defines a transition kernel $P$ by $p(x,y)=q(x^{-1} y)$ for all $x,y\in \Gamma$.

We make the following standing assumptions.
\begin{ass}\label{ass}~
\begin{itemize}
\item The underlying Galton-Watson process is supercritical, i.e.~$m>1.$  Furthermore, we assume that $\mu_{0}=0$ and $\mu_{1}>0$.
\item  The driving measure $q$ of the random walk on $G$ is symmetric, i.e.~$q(s)=q(s^{-1})$ for all $s\in S$, and satisfies $supp(q)=S\cup \{e\}$.
\end{itemize}
\end{ass}
These assumptions are to some extend chosen to improve the presentation. 
The  assumptions that are really necessary are that $m>1$ and that the driving measure $q$ is symmetric. 
 
 The spectral radius, $\rho=\rho(P)=\limsup_{n\to\infty} (p^{(n)}(e,e))^{1/n}$, is a crucial quantity in the study of BRWs:  a BRW on a Cayley graph is transient  if and only if $m\rho\leq 1$. This is a consequence of the classification of recurrent groups and Kesten's amenability criterion,  see also \cite{gantert:04} for an alternative proof. We speak of a critical BRW if $m \rho=1$.

It was shown in \cite{benjamini:11} that the trace of a transient BRW on a Cayley graph $G$ a.s.~is transient for simple random walk but recurrent for BRW. Therefore, the trace of a transient BRW is a.s.~a proper subgraph of $G$. It is believed that the trace shares many properties with infinite percolation clusters in the non-unicity phase.
In the case of BRWs on free groups (or regular trees) it even turns out that the law of the trace of a BRW is the law of an infinite cluster of some invariant percolation, see \cite{BLS:12, benjamini:11}. However, the situation is not as clear for other Cayley graphs, especially one-ended Cayley graphs.  

This note is devoted to the following property of invariant percolation and its analogue for the trace of BRWs.
For every invariant weakly insertion-tolerant percolation process on a non-amenable quasi-transitive unimodular graph $G$ that has a.s.~infinitely many clusters, we have that a.s.~every infinite cluster has continuum many ends, no isolated end, and is transient for simple random walk, see Theorem 8.32  in \cite{lyons:book}. 
 
On groups with infinitely many ends, there are various ways to see that the trace has infinitely many ends. Even the Hausdorff dimension can be calculated, see \cite{hueter:00} for free groups and  \cite{CGM:12} for free products of groups. Note that their approach seems to carry over, using Stalling's splitting theorem, to all groups with infinitely many ends. 

As pointed out in \cite{CaRo:14} there is an elegant argument using symmetry that 
the trace of a subcritical BRW, i.e.~$m<1/\rho$, has infinitely many ends. This argument extends to the critical case if $\sum_{n\geq 1} n m^{n} p^{(n)}(e, e)<\infty$. However, this criterion does not apply to the most interesting Cayley graphs like Gromov hyperbolic groups and the following conjecture remains open.

\begin{conj}[I.~Benjamini]
Let $G$ be any non-amenable vertex transitive graph and assume the transition kernel $P$ of the underlying
random walk  to be symmetric.
 Then the trace of a transient BRW has infinitely many ends.  
\end{conj}
\begin{rem}
The assumption of symmetry is crucial, since there are non-symmetric driving measures that induce one-ended traces, see \cite{CaRo:14}.
\end{rem}
We answer this conjecture affirmative for BRWs on planar hyperbolic Cayley graphs. 
\begin{thm}\label{thm:main}
Assume Assumption \ref{ass} holds. The trace of a transient BRW on a planar hyperbolic Cayley graph has a.s.~continuum many ends and no isolated end.
\end{thm}

The proof uses the fact that the trace of a BRW is a unimodular random graph and therefore gives rise to the application of the generalized Mass-Transport principle. 
This property is used to prove that a.s.~every trace that has at least $3$ ends has no isolated end,  see Proposition \ref{prop:atleast3}.   The main issue will be to prove that the trace has a.s.~no isolated end, see Proposition \ref{prop:noisolatedend}, and hence infinitely many ends. The proof of this proposition uses crucially the planarity of the Cayley graph, a recent result in \cite{Go:13} that Ancona's Inequality still holds true for the Green function at the radius of  convergence, and Proposition \ref{prop:atleast3}.

In Section \ref{sec:preliminaries} we give some background on unimodular random graphs (URG) and random walks on hyperbolic groups. We believe that the information given there is sufficient to follow the proof of Theorem \ref{thm:main} in Section \ref{sec:proof}. However, for readers who are not familiar with the concept of URGs and of random walks on hyperbolic groups before, it might be useful to consult some of the references given along Section \ref{sec:preliminaries}.

\section{Preparations}\label{sec:preliminaries}

\subsection{Definition and preliminaries}
We use the standard notation for a locally finite graph $G=(V,E)$: $V$ is the set of vertices, $E$ is the set of edges, and we write $x\sim y$ if $(x,y)\in E$. The distance between two vertices is the length of a shortest path between these vertices and will be denoted by $d(\cdot,\cdot)$. We write $(G,\orig)$ for a rooted graph with root $\orig$. 

Let $\Gamma$ be a finitely generated group with group identity $e$;   group operations are written  multiplicatively. The group $\Gamma$ together with some finite symmetric generating set $S$ induces a Cayley graph $G=G(\Gamma,S)$ whose vertex set equals $\Gamma$ and $x\sim y$ if and only  if $x^{-1}y\in S$. The group identity $e$ of $\Gamma$ will be identified with the root $\orig$ of the Cayley graph $G$.

Let $q$ be a  probability measure on the generating set $S\cup\{e\}$. The corresponding random walk $(S_{n})_{n\geq 0}$ on $G$ is a Markov chain with transition probabilities 
$p(x,y)=q(x^{-1}y)$ for $x,y\in \Gamma$.   Equivalently, the random walk  (starting in $x$) can be described as
$S_n=x X_1\cdots X_n,~n\geq1,$ 
where the $X_i$'s are i.i.d.~random variables with distribution $q$.

Besides the definition of BRWs given in Section \ref{sec:intro} there is another powerful  description of BRWs. This definition is based on the concept of tree-indexed random walks introduced in \cite{benjamini:94}.   Let $(\T,\ro)$ be a rooted infinite tree. The tree-indexed random walk  can be described as a marking  (or  labelling) of the rooted tree $(\T,\ro)$. For any vertex $v\neq \ro$ denote by $v^{-}$ the neighbour of $v$ closest to $\ro$. Label the edges of $\T$ with i.i.d.~random variables $X_v$'s with distribution $q$; the random variable $X_v$ is the label of the edge $(v^-,v)$. These labels correspond to the steps of the tree-indexed walk and  the positions of ``particles'' are given by  $S_v=x \cdot \prod_{i=1}^{n} X_{v_i}$ where $\langle v_0=\ro, v_1, \ldots, v_n=v\rangle$ is the unique geodesic from $\ro $ to $v$ at level $n$.

A tree-indexed random walk becomes a BRW if the underlying tree $\T$ is a realization of a Galton-Watson process. We call  $\T$ the family tree  of the BRW.  

\subsection{Unimodular random graphs}
In this note we only give the essentials needed for our proofs. We invite the reader to consult \cite{benjamini:11, Mueller:14} for more details on the connection between BRWs and unimodular random graphs, and \cite{AL:07} for a more general introduction to the concept of URGs.


 A rooted isomorphism between two rooted graphs $(G,\orig)$ and $(G',\orig')$ is an  isomorphism of $G$ onto $G'$ which maps $\orig$ to $\orig'$. We denote by $\G_*$ the space of isomorphism classes of rooted graphs and  write $[G,\orig]$ for the equivalence class that contains $(G,\orig)$.  In the same way one defines the space $\G_{**}$ of isomorphism classes of graphs with an ordered pair of distinguished vertices. That is, $(G_1,o_1,o_2)$ and $(G_2,o'_1,o'_2)$ are isomorphic if and only if there is an isomorphism from $G_1$ onto $G_2$ which maps $o_1$ to $o_1'$ and $o_2$ to $o_2'$. The spaces $\G_{*}$ and $\G_{**}$ can be  equipped with  metrics that turn  them into separable and complete metric spaces.
 
  A Borel probability measure $\nu$ on $\G_*$ is called \emph{unimodular} if it obeys the Mass-Transport Principle (MTP): for all Borel function $f: \G_{**}\to[0,\infty]$, we have
\begin{equation}\label{eq:defMTP}
\int \sum_{x\in V} f(G,\orig,x)d\nu([G,\orig])= \int \sum_{x\in V} f(G,x,\orig) d\nu([G,\orig]).
\end{equation}
Realizations of unimodular measures are called \emph{unimodular random graphs}.

An important class of unimodular measures arises from Galton-Watson processes. The Galton-Watson tree is defined inductively: start with one vertex, the root $\ro$ of the tree. Then the number of offspring of each particle (vertex) is distributed according to $\mu$. Edges are between vertices and their offspring. We denote by \GW~the corresponding measure on the space of rooted trees.  In this
construction the root clearly plays a special role and \GW~is not unimodular.  However, if we bias the distribution such that the probability that the root has degree $k+1$ is proportional to $\frac{\mu_{k}}{k+1}$ we obtain a unimodular measure \UGW. When we use the \UGW~measure instead of the standard \GW~measure to define the family tree of the BRW we denote the BRW by UBRW.

Due to the description of the BRW as a tree-indexed random walk the unimodularity of \UGW~caries over to the trace: the trace of a UBRW on a Cayley graph is a unimodular random graph, see Theorem 3.7 in \cite{benjamini:11}.
This property makes the UBRW more natural to consider than the original BRW and we  will prove Theorem \ref{thm:main} for UBRWs. However, it is not difficult to see that it then also holds true for BRWs.

\subsection{Ends of graphs}
Consider a locally finite graph $G=(V,E)$. A \emph{ray} is a sequence $\pi=\langle x_{0},x_{1},\ldots\rangle$ of distinct vertices such that $x_{i}\sim x_{i+1}$ for all $i\geq 0$. For any  finite set $F$ of vertices  we consider its complement $G\setminus F$, which is the graph induced by the vertex set $V\setminus F$. This graph consists of finitely many connected components. Every ray $\pi$ must have all but finitely many points in exactly one component; we say that $\pi$ \emph{ends up} in that component. Two ends are  \emph{equivalent} if they end up in the same connected component for  all choices of $F$. \emph{Ends} are equivalence classes of rays and we denote by $\vartheta G$ the set of ends. Let $F$ be a finite vertex set and $C$ be some component of $G\setminus F$. We write $\vartheta C$ for the set of ends whose rays end up in $C$. The space of ends $\vartheta G$ can be equipped with a discrete topology in the following way. For any finite set $F$ and any end $w$ there is precisely one component of $G\setminus F$ whose completion contains $w$. Varying $F$ yields a neighbourhood base for $w$.  An \emph{isolated end} is an end that is isolated in this topology.

\subsection{Hyperbolic groups and random walks}
 From now on we assume that the underlying group $\Gamma$ is hyperbolic and the generating set $S$ induces a planar Cayley graph $G=G(\Gamma, S)$.  Both assumptions, hyperbolicity and planarity, are crucial.
 
Let us first collect several classic facts about hyperbolic groups and random walks; we refer to the survey \cite{Cal:13}  for an excellent introduction. 
 An elementary hyperbolic group is either finite or has two ends. 
We will focus on the case of non-elementary hyperbolic groups since  
random walks on them  are transient. 
Define the \emph{Green functions} 
\begin{equation*}
G_{r}(x,y)=\sum_{n=0}^{\infty} p^{(n)}(x,y) r^{n},\quad x,y\in\Gamma,
\end{equation*}
for all $r\leq R:=1/\rho$.  It is proved in   \cite{Go:13} that for finite range random walks on  hyperbolic groups Ancona's Inequalities hold up to the radius of  convergence: there exists some $C>0$ such that for any $x,z\in \Gamma$ and for any $y$ on a geodesic segment from $x$ to $z$ we have 
\begin{equation*}
G_{r}(x,z)\leq C G_{r}(x,y) G_{r}(y,z) \quad \forall r\in[1,R].
\end{equation*}

Symmetry of the random walk implies, see Lemma 2.1 in \cite{GoLa:13},  that 
\begin{equation*}
\lim_{d(\orig,x)\to\infty} G_{R}(\orig,x)=0.
\end{equation*}

Eventually we obtain, see Lemma 2.2 in \cite{GoLa:13},  that the Green functions of the random walk decay exponentially, that is, there exist some constants $C_{1}$ and $\varrho<1$ such that for all $x,y\in \Gamma$ 
\begin{equation}\label{eq:expdecay}
G_{R}(x,y)\leq C_{1} \varrho^{d(x,y)}.
\end{equation}

In the case of planar Cayley graphs, the Cayley $2$-complex is the $2$-complex such that the one-skeleton is given by the Cayley graph $G$ and the $2$-cells are bounded by loops in $G$. The $2$-complex is homeomorphic to the hyperbolic disc and it can be endowed with an orientation. This orientation is used implicitly at several points in the proof of Proposition \ref{prop:noisolatedend}. Moreover, the Gromov hyperbolic boundary $\partial G$ can be identified with the unit circle.

\section{Proof of Theorem \ref{thm:main}}\label{sec:proof}
The following result holds true for traces of transient BRWs on Cayley graphs (or even URGs). Its proof is an  adaption of the one for invariant percolation, see Proposition 8.33 in \cite{lyons:book}. As the arguments are short we allow us to present the details. We stick as close as possible to the notations in  \cite{lyons:book}.
\begin{prop}\label{prop:atleast3}
Consider the trace of a transient symmetric UBRW on a Cayley graph. Almost surely every  trace that has at least $3$ ends has no isolated end.
\end{prop}
\begin{proof}
For each $n\in \N$ let $A_{n}$ be the union of all vertex sets $A\subset \Tr$ such that  the diameter $diam(A)\leq n$ (in the metric in $\Tr$) and $\Tr\setminus A$ has at least $3$ infinite components. If $\Tr$ has at least $3$ ends then $A_{n}\neq \emptyset$ for all but finitely many $n$. We assume from now on that $\Tr$ has at least $3$ ends.

Fix some $n\in \N$.  For any vertex $x$ in $\Tr$ let 
\[C(x)=\{y\in A_{n}: d(x,y)=\min_{z\in A_{n}} d(x,z)\}\] 
be the set of vertices in $A_{n}$ that are closest to $x$ in the metric in $\Tr$. We can define the Borel function $F: \G_{**}\to[0,\infty]$:
\[
F(\Tr,x,y) = 
\begin{cases}
\frac{1}{|C(x)|} \1_{\{y\in C(x)\}}, & \textrm{if } A_n\neq \emptyset,\\
0, & \textrm{otherwise}.
\end{cases}
\]
The function $F$ is well-defined since $F(\Tr,\orig,x)$ is invariant under isomorphisms. Since the law $\nu$ of the rooted trace $[\Tr,\orig]$ is unimodular we can apply the MTP to obtain that the expected mass received by $\orig$  is at most $1$:
\begin{equation*}
\int \sum_{x\in \Tr} F(\Tr,x,\orig)d\nu([\Tr,\orig])=\int \sum_{x\in \Tr} F(\Tr,\orig,x)d\nu([\Tr,\orig])\leq 1.
\end{equation*}

Assume that $\xi$ is an isolated end of $\Tr$ and let us show that this leads to  a contradiction.  There exists some finite set  $F$ such that the connected component $U$ of $G\setminus F$ whose completion contains  $\xi$ satisfies $U\cap A_{n}=\emptyset$. Moreover, there exists a finite set of vertices $B$ such that all paths from $U$ to all other connected components of $G\setminus F$ must pass through $B$. Then a subset of  vertices of  $B$  gets all the mass from all the vertices in $U$.  As $U$  contains infinitely many vertices the set $B$ receives infinite mass. 
 Using the property that ``everything shows at the root'', see \cite[Lemma 2.3]{AL:07}, we obtain that $\int \sum_{x\in \Tr} F(\Tr,x,\orig)d\nu([\Tr,\orig])=\infty$, a contradiction. Hence, almost surely every trace with $A_{n}\neq \emptyset$ does not have isolated ends. Since this holds for all but finitely many $n\in\N$, we obtain that almost surely a trace with isolated ends can have at most two ends.
 \end{proof}

\begin{prop}\label{prop:noisolatedend}
Assume Assumption \ref{ass} holds. The trace of a transient UBRW on a planar hyperbolic Cayley graph has a.s.~no isolated end.
\end{prop}
\begin{proof}
We start with some preparations. Since the Gromov boundary $\partial G$ can be identified with a circle there exist infinite geodesics  $\gamma_{1},\gamma_{2}$ and $\gamma_{3}$ starting from $\orig$ 
with distinct boundary points $\xi_{1},\xi_{2}$ and $\xi_{3}$. 
Denote by $\gamma=\bigcup_{i=1}^{3}\{\gamma_{i}\}$, where $\{\gamma_{i}\}$ is the set of vertices in $\gamma_{i}$.  We define $\partial G_{ij}$ to be the part of the Gromov boundary that is between $\xi_{i}$ and $\xi_{j}$ and $S_{ij}$ to be the set of vertices between $\gamma_{i}$ and $\gamma_{j}$. Let $K$ be some large positive constant to be chosen later. Geodesics in hyperbolic groups either converge to the same boundary point or diverge exponentially.  Hence,  for $K$ sufficiently large there exist  $x_{1}\in S_{12}, x_{2}\in S_{23}$, and $x_{3}\in S_{31}$ such that $d(x_{i},\gamma)=K$ for all $i\in\{1,2,3\}$. 

Let $i\in\{1,2,3\}$ and consider  the sphere $S(x_{i},n)$ of radius $n$ around $x_{i}$ in the Cayley graph metric. Since the $\gamma_{i}$'s are geodesics and triangles are thin we have that there exists some constant $C_{2}$ such that $|S(x_{i},n)\cap \gamma| \leq C_{2} n$. Now, start a BRW in $x_{i}$ and denote by $\Tr_{i}$ the trace of the BRW started in $x_{i}$.  Using Markov's Inequality and Inequality (\ref{eq:expdecay}) we obtain that
\begin{equation}\label{eq:tozero}
\P[|\Tr_{i}\cap \gamma| \geq 1]\leq \E[|\Tr_{i}\cap \gamma| ] \leq  \sum_{y\in \gamma} G_{m}(x,y)\leq  \sum_{n\geq K}C_{1} C_{2} n \varrho^{n},
\end{equation}
which tends to $0$ as $K$ tends to infinity.

We assume now that $\P[\Tr \mbox{ has an isolated end}]=c>0$ and show that this yields a contradiction. Inequality (\ref{eq:tozero}) allows us to choose $K$ sufficiently large such that for $i\in\{1,2,3\}$
\begin{equation*}
\P[\Tr_{i} \mbox{ has an isolated end}, |\Tr_{i}\cap \gamma|=0 ]>0.
\end{equation*}
We start the UBRW in $\orig$ and condition the UBRW on the event that at time $N=\max\{d(\orig, x_{i}): i\in\{1,2,3\}\}$ in each of the vertices $x_{1},x_{2}$ and $x_{3}$ there is exactly one particle and no particle elsewhere. This is possible since we assume that  $\mu_{1}>0$ and $q(e)>0$. So what happens at times $n\geq N$ has the same distribution as we start three independent BRWs in $x_{1},x_{2}$, and $x_{3}$.  
Eventually, using the planarity of $G$ we get that $\Tr$ has with positive
 probability at least three
 distinct ends including at least one isolated end, which yields a contradiction to Proposition \ref{prop:atleast3}.
\end{proof}
It remains to show that the trace has a.s.~continuum many ends.
\begin{cor}
Assume Assumption \ref{ass} holds. The trace of a transient UBRW on a planar hyperbolic Cayley graph has a.s.~continuum many ends.
\end{cor}
\begin{proof}
Due to Proposition \ref{prop:noisolatedend} the trace $\Tr$ must have infinitely many
ends because otherwise each end would be isolated. Moreover,  each infinite connected component of $\Tr \setminus B(\orig,n)$ must contain at least two ends; otherwise, such a component would contain an isolated end. Thus, the number of ends is at least of order $|2^{\mathbb{N}}|$, which proves the claim since $\Tr$ is of bounded degree.
\end{proof}

\subsection*{Acknowledgment}
The authors thank Elisabetta Candellero and Matthew Roberts for comments on a first version of this note. The research was supported by the exchange programme Amadeus-Amad\'ee $31473$TF.
\bibliographystyle{plain}
\bibliography{bib}

\def\cprime{$'$}
\begin{thebibliography}{10}

\bibitem{AL:07}
D.~Aldous and R.~Lyons.
\newblock Processes on unimodular random networks.
\newblock {\em Electron. J. Probab.}, 12:no. 54, 1454--1508, 2007.

\bibitem{BLS:12}
I.~Benjamini, R.~Lyons, and O.~Schramm.
\newblock Unimodular random trees.
\newblock {\em Preprint, http://arxiv.org/abs/1207.1752}, 2012.

\bibitem{benjamini:11}
I.~Benjamini and S.~M{\"u}ller.
\newblock On the trace of branching random walks.
\newblock {\em Groups Geom. Dyn.}, 6(2):231--247, 2012.

\bibitem{benjamini:94}
I.~Benjamini and Y.~Peres.
\newblock Markov chains indexed by trees.
\newblock {\em Ann. Probab.}, 22(1):219--243, 1994.

\bibitem{Cal:13}
D.~Calegari.
\newblock The ergodic theory of hyperbolic groups.
\newblock {\em Contemp. Math}, pages 15--52, 2013.

\bibitem{CGM:12}
E.~Candellero, L.~A. Gilch, and S.~M{\"u}ller.
\newblock Branching random walks on free products of groups.
\newblock {\em Proc. Lond. Math. Soc. (3)}, 104(6):1085--1120, 2012.

\bibitem{CaRo:14}
E.~Candellero and M.~I. Roberts.
\newblock The number of ends of critical branching random walks.
\newblock {\em http://arxiv.org/abs/1401.0429}, 2014.

\bibitem{gantert:04}
N.~Gantert and S.~M\"uller.
\newblock The critical branching {M}arkov chain is transient.
\newblock {\em Markov Process. and Rel. Fields.}, 12:805--814, 2007.

\bibitem{Go:13}
S.~Gou{{\"e}}zel.
\newblock Local limit theorem for symmetric random walks in {G}romov-hyperbolic
  groups.
\newblock {\em J. Amer. Math. Soc.}, 27(3):893--928, 2014.

\bibitem{GoLa:13}
S.~Gou{{\"e}}zel and S.~P. Lalley.
\newblock Random walks on co-compact {F}uchsian groups.
\newblock {\em Ann. Sci. {\'E}c. Norm. Sup{\'e}r. (4)}, 46(1):129--173 (2013),
  2013.

\bibitem{hueter:00}
I.~Hueter and S.~P. Lalley.
\newblock Anisotropic branching random walks on homogeneous trees.
\newblock {\em Probab. Th. Rel. Fields}, 116(1):57--88, 2000.

\bibitem{lyons:book}
R.~Lyons, with Y.~Peres.
\newblock {\em Probability on Trees and Networks}.
\newblock Cambridge University press, In preparation. Current version available
  at {\tt http://mypage.iu.edu/\string~rdlyons/}.

\bibitem{Mueller:14}
S.~M\"uller.
\newblock Interacting growth processes and invariant percolation.
\newblock {\em Ann. Appl. Prob.}, to appear.

\end{thebibliography}
\bigskip

\noindent\begin{minipage}{0.48\textwidth}
Lorenz A.~Gilch\newline
Dep. of Math. Structure Theory,\newline
Graz University of Technology\newline
Steyrergasse 30,\newline
8010 Graz, Austria\newline
\texttt{gilch@TUGraz.at}
\end{minipage}
\hfill
\begin{minipage}{0.48\textwidth}
Sebastian M\"{u}ller\newline
Aix Marseille Universit\'e\newline
CNRS  Centrale Marseille\newline 
I2M\newline
UMR 7373\newline 13453 Marseille France\newline
\texttt{mueller@cmi.univ-mrs.fr}
\end{minipage}

\end{document}